\documentclass{amsart}
\usepackage{amsmath,amssymb,amsthm,
verbatim}%for the comment environment
\newtheorem{theorem}{Theorem}
\newtheorem{proposition}{Proposition}

\newtheorem{corollary}{Corollary}

\theoremstyle{definition}
\newtheorem{definition}{Definition}

\begin{document}
\title{On a Question of Glasby, Praeger, and Xia in Characteristic $2$}
\author{Michael~J.~J.~Barry}
\address{Department of Mathematics\\
Allegheny College\\
Meadville, PA 16335}
\email{mbarry@allegheny.edu}
\thanks{This work was done when the author was on sabbatical and he thanks Allegheny College for its support.}


\begin{abstract}
Recently, Glasby, Praeger, and Xia asked for necessary and sufficient conditions for the `Jordan partition' $\lambda(m,n,p)$ to be standard.  Previously we gave such conditions when $p$ is any odd prime.  Here we give such conditions when $p=2$.  Our main result is that $\lambda(m,n,2)$ is never standard for $4 \leq m \leq n$.
\end{abstract}
\maketitle
For positive integers $m$ and $n$ with $m \leq n$, the {\bf Jordan partition} $\lambda(m,n,p)=(\lambda_1, \dots, \lambda_m)$ is a partition of $m \cdot n$ into $m$ parts that arises when writing the tensor product of two Jordan matrices over a field of characteristic $p$ as a direct sum of  Jordan matrices.  See~\cite{GPX2014} for more background on this subject.  The Jordan partition  $\lambda(m,n,p)$ is {\bf standard} if{f} $\lambda_i=m+n-2 i+1$ for $i=1$, \dots $m$.

In response to a question asked in~\cite{GPX2014}, we gave necessary and sufficient conditions for a Jordan partition $\lambda(m,n,p)$ to be standard when $p$ is an odd prime in ~\cite{arXivB2014}.   Here we handle the case of $p=2$.
\begin{theorem}\label{Th:Char2}
Let $m$ and $n$ be positive integers with $m \leq n$.  Then $\lambda(m,n,2)$ is standard if{f} one of the following three conditions holds:
\begin{enumerate}
\item $m=1$
\item $m=2$ and $n$ is odd
\item $m=3$ and $n=6 +4 k$ for a nonnegative integer $k$
\end{enumerate}
\end{theorem}

Our main tool in establishing this result is the composition of $m$ that we associate to $\lambda(m,n,p)$.  For now we will allow $p$ to be any prime.
\begin{definition}
Suppose that the Jordan partition
\[\lambda(m,n,p)=(\overbrace{\lambda_1,\dots,\lambda_1}^{m_1},\overbrace{\lambda_2,\dots,\lambda_2}^{m_2},
\dots,\overbrace{\lambda_r,\dots,\lambda_r}^{m_r})
=(m_1 \cdot \lambda_1, \dots, m_r \cdot \lambda_r)\]
where $\lambda_1 > \dots > \lambda_r >0$ and $\sum_{i=1}^r m_i=m$.  Denote the composition $(m_1,\dots,m_r)$ of $m$ by $c(m,n,p)$. 
\end{definition}

Note that if $\lambda(m,n,p)$ is standard, then $c(m,n,p)=(m \cdot 1)$.  The converse is true.   But in fact much more is true:  $\lambda(m,n,p)$ is completely determined by $n$ and $c(m,n,p)$, as we will prove in Proposition~\ref{Pr:CtoLambda}.

We now record two properties of $c(m,n,p)$ which follow from Theorem 4 of~\cite{GPX2014}. 
\begin{theorem}\label{Th:PerDual}
Suppose that $m \leq p^t$.
\begin{enumerate}
\item Then $c(m,n,p)=c(m,n+p^t,p)$ for every integer $n \geq m$.
\item Then $c(m,p^t+i,p)=r(c(m,2 p^t-i,p))$ for every integer $i$ in $[0, p^t]$.
\end{enumerate}
\end{theorem}

In~\cite{B2011}, we gave a recursive definition of the sequence $s_p(m,n)$, where $m \leq n$, of $m+n$ integers in six mutually exclusive and exhaustive cases.  In~\cite{arXivB2014}, we related it to the Jordan partition $ \lambda(m,n,p)$ by noting that
\[\lambda(m,n,p)=(s_p(m,n)(1),\dots,s_p(m,n)(m)),\] that is, $\lambda(m,n,p)$ is the subsequence of the first $m$ elements of $s_p(m,n)$.  In light of the definition of $s_p(m,n)$, $c(m,n,p)$ satisfies the following conditions in the corresponding cases:
\begin{enumerate}
\item $c(m,n,p)=(m+n-p^{k+1}) \oplus c(p^{k+1}-n,p^{k+1}-m,p)$
\item $c(m,n,p)=(c+d-p^k) \oplus  c((a+b+1)p^k-n,(a+b+1)p^k-m,p)$
\item $c(m,n,p)=c_1 \oplus c((a+b)p^k-n,(a+b)p^k-m,p)$ where\\
$c_1=c(\min(c,d),\max(c,d),p) \oplus (|c-d|)\oplus r(c(\min(c,d),\max(c,d),p))$
\item $c(m,n,p)=c(m,b p^k+d,p)=r(c(m,b p^k-d,p))$
\item $c(m,n,p)=c(c,b p^k,p)=(m)$
\item $c(m,n,p)=c(a p^k,b p^k,p)=(p^k) \oplus c((a-1)p^k,(b-1)p^k,p)$
\end{enumerate}
where $\oplus$ is the operation of concatenating two sequences into one and $r(s)$ denotes the reverse of the sequence $s$.  See~\cite{arXivB2014} or~\cite{B2011} for more details of the six cases.

Of course $c(m,n,p)$ determines $m$ since $\sum_{i=1}^r m_i=m$.  Next we show that $n$ and $c(m,n,p)$ determine $\lambda(m,n,p)$.
\begin{proposition}\label{Pr:CtoLambda}
If $c(m,n,p)=(m_1,\dots,m_r)$, then
\[\lambda_i=\frac{1}{m_i}\sum_{k=m_1+\dots+m_{i-1}+1}^{m_1+\dots+m_i}(m+n-2k+1)
=n+\sum_{k=i+1}^r m_k-\sum_{k=1}^{i-1} m_k\]
when $1 \leq i \leq r$.
\end{proposition}
\begin{proof}
Note that
\[n+\sum_{k=i+1}^r m_k-\sum_{k=1}^{i-1} m_k=m+n-2 \sum_{k=1}^{i-1}-m_i,\]
and we will use the second form at points of the proof.

Let $k$ be the unique nonnegative integer such that $p^k \leq n<p^{k+1}$.    Write $n= b p^k+d$ with $0<b<p$ and $0 \leq d <p^k$.  Write $m=a p^k +c$ with $0 \leq a<p$ and $0 \leq c <p^k$.  Note that $a+c>0$.

We proceed by induction on $m+n$.  The result is true for $(m,n)=(1,1)$.  Assume the result is true for all $(m',n')$ with $m'+n'<m+n$.

Case 1: $m+n>p^{k+1}$.  Then $m_1=m+n-p^{k+1}$, $\lambda_1=p^{k+1}$, and $c(p^{k+1}-n,p^{k+1}-m,p)=(m_2,\dots,m_r)$.  Note that $n+m-m_1=p^{k+1}=\lambda_1$, that is,  the formula holds for $i=1$.  By inductive assumption, when $2 \leq i \leq r$,
\begin{align*}
\lambda_i&=(p^{k+1}-m) +(m_{i+1}+\dots+m_r)-(m_2+\dots+m_{i-1})\\
&=(n-m_1)+(m_{i+1}+\dots+m_r)-(m_2+\dots+m_{i-1})\\
&=n+(m_{i+1}+\dots+m_r)-(m_1+m_2+\dots+m_{i-1}),
\end{align*}
showing the formula holds $i=2$, \dots,$r$.

Case 2: $m+n \leq p^{k+1}$ but $c+d>p^k$.  Then $m_1=c+d-p^k$, $\lambda_1=(a+b+1)p^k$, and $c((a+b+1)p^k-n,(a+b+1) p^k-m,p)=(m_2,\dots,m_r)$.  Note that $n+m-m_1=(a+b+1) p^k =\lambda_1$, that is, the formula holds for $i=1$.  Verification that the formula for $\lambda_i$ holds when $2 \leq i \leq m$ is similar to Case 1 using the fact that $(a+b+1)p^k-m=n-m_1$.

Case 3: $m+n \leq p^{k+1}$, $1 \leq c+d \leq p^k$, and $a>0$.  Suppose that \[c(\min(c,d),\max(c,d),p)=(m_1',\dots,m_s')\] and \[\lambda(\min(c,d),\max(c,d),p)=(m_1' \cdot \lambda_1', \dots, m_s' \cdot \lambda_s').\]  So $m_i'=m_i$ and $\lambda_i'=\lambda_i-(a+b) p^k$ when $1 \leq i \leq s$.  Let $t=2 s+\epsilon$ where $\epsilon=1$ if $c \neq d$ and $\epsilon=0$ if $c = d$.  Then $c((a+b)p^k-n,(a+b)p^k-m,p)=(m_{t+1},\dots,m_r)$ and
\[(m_1,\dots, m_t)=\begin{cases}
(m_1',\dots,m_s',m_s', \dots,m_1'), & \text{if $c=d$}\\
(m_1',\dots,m_s',|c-d|,m_s', \dots,m_1'), & \text{if $c \neq d$}
\end{cases}\]

By inductive assumption, when $1 \leq i \leq s$,
\[
\lambda_i'=c+d-2 \sum_{k=1}^{i-1}m_k'-m_{i}'
=c+d-2 \sum_{k=1}^{i-1}m_k-m_{i}.
\]

Thus, when $1 \leq i \leq s$,
\[
\lambda_i=\lambda_i'+(a+b) p^k=m+n-2 \sum_{k=1}^{i-1}m_k-m_{i}\]
and the formula holds.

Assume that $c \neq d$, so $\lambda_{s+1}=0+(a+b) p^k$ and $m_{s+1}=|c-d|$.  Now
\begin{align*}
\lambda_{s+1}&=(a+b)p^k\\
&=(a+b)p^k+c+d-2 \min(c,d)-|c-d|\\
&=m+n-2 \sum_{k=1}^{s}m_k-m_{s+1}
\end{align*}
and the formula holds for $i=s+1$.

We now consider the case of $t-s+1 \leq i \leq r$.  Then $m_i=m_{t+1-i}$ and $\lambda_i=-\lambda_{t+1-i}'+(a+b) p^k$.  Thus
{\allowdisplaybreaks
\begin{align*}
\lambda_i&=-\lambda_{t+1-i}'+(a+b) p^k\\
&=(a+b) p^k-(c+d)+2 \sum_{k=1}^{t-i}m_k+m_{t-i+1}\\
&=m+n-2(c+d)+2 \sum_{k=1}^{t-i}m_k+m_{t-i+1}\\
&=m+n-2 \min(c,d)-|c-d| -(c+d)+2 \sum_{k=1}^{t-i}m_k+m_{t-i+1}\\
&=m+n-2 \sum_{k=1}^s m_i -|c-d| -(c+d)+2 \sum_{k=1}^{t-i}m_k+m_{t-i+1}\\
&=m+n-2 \sum_{k=1}^s m_i -|c-d|-|c-d|-2 \sum_{k=1}^s m_k+2 \sum_{k=1}^{t-i}m_k+m_{t-i+1}\\
&=m+n-2 \sum_{k=1}^s m_i -|c-d|-|c-d|-2 \sum_{k=t-i+1}^s m_k+m_{t-i+1}\\
&=m+n-2 \sum_{k=1}^s m_i -2|c-d|-2 \sum_{k=t-s+1}^i m_k+m_{i}\\
&=m+n-2 \sum_{k=1}^s m_i -2|c-d|-2 \sum_{k=t-s+1}^{i-1} m_k-m_{i}\\
&=m+n-2\sum_{k=1}^{i-1}m_k-m_i.
\end{align*}}

By inductive assumption, when $t+1 \leq i \leq m$,
\begin{align*}
\lambda_i&=(a+b) p^k-m+(m_{i+1}+\dots+m_r)-(m_{t+1}+\dots+m_{i-1})\\
&=b p^k-c+(m_{i+1}+\dots+m_r)-(m_{t+1}+\dots+m_{i-1})\\
&=b p^k+d+(m_{i+1}+\dots+m_r) -(c+d)-(m_{t+1}+\dots+m_{i-1})\\
&=n+(m_{i+1}+\dots+m_r)-(2 \min(c,d)+|c-d|)-(m_{t+1}+\dots+m_{i-1})\\
&=n+(m_{i+1}+\dots+m_r)-(m_1+\dots+m_t)-(m_{t+1}+\dots+m_{i-1}).
\end{align*}
Thus the formula for $\lambda_i$ holds when $t+1 \leq i \leq r$.

Case 4: $m+n \leq p^{k+1}$, $1 \leq c+d \leq p^k$, $a=0$ (so $m=c$), and $d>0$.  Then $c(m,n,p)=r(c(m,b p^k-d,p))$.
Suppose that $c(m,b p^k-d,p)=(m_1',\dots,m_r')$ and $\lambda(m,b p^k-d,p)=(m_1' \cdot \lambda_1',\dots,m_r' \cdot \lambda_r')$.  Hence $m_i=m_{r+i-i}'$ and $\lambda_i=\lambda_{r+1-i}'+2 b p^k$.  Then, by inductive assumption,
\[\lambda_i'=(b p^k-d)+(m_{i+1}'+\dots+m_r')-(m_1'+\dots+m_{i-1}')\]
and so
\begin{align*}
\lambda_{r+1-i}'&=(b p^k-d)+(m_{r+2-i}'+\dots+m_r')-(m_1'+\dots+m_{r-i}')\\
&=(b p^k-d)+(m_{i-1}+\dots+m_1)-(m_r+\dots+ m_{i+1}).
\end{align*}
Therefore
\begin{align*}
\lambda_i&=-\lambda_{r+1-i}'+2 b p^k \\
&=2 b p^k-b p^k+d-(m_{i-1}+\dots+m_1)+(m_r+\dots+ m_{i+1})\\
&=n+(m_{i+1}+\dots+m_r)-(m_1+\dots+m_{i-1})
\end{align*}
when $1 \leq i \leq r$.

Case 5: $m+n \leq p^{k+1}$, $1 \leq c+d \leq p^k$, $a=0$, and $d=0$, so $(m,n)=(c,b p^k)$.  In this case $r=1$, $m_1=m$, and $\lambda_1=n$. Therefore $n+m-m_1=\lambda_1$ and the formula is valid.

Case 6: $m+n \leq p^{k+1}$, $c=d=0$, so $0<a<a+b \leq p$.  Then $m_1=p^k$, $\lambda_1=(a+b-1)p^k$, and $c((a-1) p^k,(b-1)p^k,p)=(m_2,\dots,m_r)$.  Because $n+m-m_1=\lambda_1$,  the formula holds for $i=1$.  Verification that the formula for $\lambda_i$ holds when $2 \leq i \leq m$ is similar to Case 1. 
\end{proof}

\begin{corollary}\label{Cor:Standard}
If $m$ and $n$ are positive integers with $m \leq n$, then $\lambda(m,n,p)$ is standard if{f} $c(m,n,p)=(m \cdot 1)$. \end{corollary}

\begin{proof}[Proof of Theorem~\ref{Th:Char2}]
It is clear that $c(1,n,2)=(1)$ and so $\lambda(1,n,2)$ is standard for all integers $n \geq 1$.
Suppose that $2^{t-1}<m \leq 2^t$.  By the periodicity result of Theorem~\ref{Th:PerDual}, it suffices to compute $c(m,n,2)$ for integers $n$ in the interval $[2^t,2^{t+1}-1]$.

First consider the case of $m=2^t$ where $t$ is a positive integer.  Then, by Case 6, $c(m,2^t)=(2^t)$.  Now we look at the case $(m,n)=(2^t, 2^t+i)$, where $1 \leq i \leq 2^t-1$.  Then we are in Case 1, so
$c(m,n,2)=(i) \oplus c(2^t-i,2^t,2)$.
Since $(2^t-i,2^t,2)$ is a Case 5 situation, $c(2^t-i,2^t,2)=(2^t-i)$.  Thus $c(m,n,2)=(i) \oplus (2^t-i)$.  We have shown that $c(2,n,2)=(1,1)$ if{f} $n$ is odd, and that $c(4,n,2)$ never equals $(1,1,1,1)$.  By Corollary~\ref{Cor:Standard}, $\lambda(2,n,2)$ is standard if{f} $n$ is odd.

Next we prove by induction on $m$ that $c(m,n,p)$ never equals $(m \cdot 1)$ for $4 \leq m \leq n$.  We have already verified this for the base case $m=4$ and  any $m =2^t \geq 4$.  We can assume that $4 \leq 2^{t-1}<m<2^t$, and write $m=2^{t-1}+j$ where $1 \leq j \leq 2^{t-1}-1$.  Then $(m,2^t)$ is a Case 5 situation, so $c(m,2^t,2)=(m) \neq (m \cdot 1)$.  When $1 \leq i \leq 2^{t-1}-j$, $(m, 2^t+i)$ is a Case 4 situation, and so $c(m,2^t+i,2)=r(c(m,2^t-i,2))$.  By our inductive assumption, $c(m,2^t-i,2) \neq (m \cdot 1)$, which implies that $c(m,n,2) \neq (m \cdot 1)$.

When $2^{t-1}-j+1 \leq i \leq 2^t-1$, $(m,2^t+i)$ is a Case 1 situation.  Therefore $c(m,2^t+i,2)=(m+i-2^t) \oplus c(2^t-i,2^t+2^{t-1}-j,2)$.  If $m+i-2^t>1$, then $c(m,2^t-i,2) \neq (m \cdot 1)$.  Assume that $m+i-2^t=1$.  Then $2^t-i=m-1 \geq 4$.  By inductive assumption, $c(2^t-i,2^t+2^{t-1}-j,2) \neq ((m-1) \cdot 1)$, which implies that $c(m,2^t+i,2) \neq (m \cdot 1)$.  This completes our proof by induction.

Finally one can check that $c(3,4,2)=(3)$, $c(3,5,2)=(1,2)$, $c(3,6,2)=(1,1,1)$, and $c(3,7,2)=(2,1)$.  As a consequence $c(3,n,2)=(1,1,1)$ if{f} $n=6 +4 k$ for some nonegative integer $k$.  By Corollary~\ref{Cor:Standard}, $\lambda(3,n,2)$ is standard if{f} $n=6 +4 k$ for some nonegative integer $k$.
\end{proof}

\end{document}